\documentclass[12pt,a4paper, reqno]{amsart}
\usepackage[utf8]{inputenc}
\usepackage{mathptmx}
\usepackage{amssymb}
\usepackage{hyperref}
\usepackage{mathrsfs}
\usepackage[mathscr]{euscript}
\usepackage{booktabs}

\usepackage{todonotes}

\usepackage[ignoreunlbld,norefs,nocites,nomsgs]{refcheck}

\usepackage{pgf, tikz}

\usepackage{fancyvrb} \RecustomVerbatimEnvironment{Verbatim}{Verbatim}
{xleftmargin=15pt, frame=single, fontsize=\small}

\usepackage{color}

\definecolor{light}{gray}{0.9}
\definecolor{medium}{gray}{0.8}

\newtheorem{theorem}{Theorem}

\newtheorem{proposition}[theorem]{Proposition}

\theoremstyle{definition}
\newtheorem{remark}[theorem]{Remark}

\numberwithin{equation}{section}

\def\ZZ{{\mathbb Z}}
\def\NN{{\mathbb N}}
\def\RR{{\mathbb R}}
\def\QQ{{\mathbb Q}}
\def\11{{\mathbb 1}}

\def\cF{{\mathcal F}}

\def\conv{\operatorname{conv}}

\def\ttt#1{\texttt{#1}}

\let\epsilon=\varepsilon

\allowdisplaybreaks

\def\Vol{\operatorname{Vol}}
\def\aff{\operatorname{aff}}
\def\Ht{\operatorname{Ht}}

\def\strut{\vphantom{\Large(} }

\begin{document}
	

\title{Polytope volume in Normaliz}

\author{Winfried Bruns}
\address{Universit\"at Osnabr\"uck, Institut f\"ur Mathematik, 49069 Osnabr\"uck, Germany}
\email{wbruns@uos.de}
\date{}
\dedicatory{To the memory of Wolmer Vasconcelos\\and to Rafael Villareal on his 70th birthday}


\begin{abstract}
We survey the computation of polytope volumes by the algorithms of Normaliz to which the Lawrence algorithm has recently been added. It has enabled us to master volume computations for polytopes from social choice in dimension $119$. This challenge required a sophisticated implementation of the Lawrence algorithm.
\end{abstract}
\maketitle

\section{Introduction}

About 20 years ago Amelia Taylor asked the author whether Normaliz \cite{Nmz} could compute polytope volumes. It was easy to include this computation goal into the triangulation based ``primal algorithm''. Since then, polytope volumes have played an important role in the development of Normaliz, and in recent years specific algorithms have been added.

Polytope volumes can be interpreted as degrees of projective toric varieties and multiplicities of monomial algebras. In 2011, Bogdan Ichim pointed out their applications in social choice. Since then they have been a driving challenge for the volume algorithms in Normaliz whose history we sketch briefly. Before going on, let us emphasize that Normaliz computes lattice normalized volumes that for rational polytopes are rational numbers. Euclidean volumes, if asked for by the user, are derived from them. 
 
In social choice, polytope volumes are interpreted as probabilities of certain paradoxa and quality measures of voting schemes. See the books by Gehrlein and Lepelley \cite{GL1}, \cite{GL2}. These applications become rapidly very difficult since the relevant polytopes explode in dimension: for  $n$ candidates they are cross-sections of cones of dimension $n!$. 
In their paper \cite[p.\ 382]{LLS} of 2008 Lepelley, Louichi and Smaoui state:
\begin{quotation}
Consequently, it is not possible to analyze
four candidate elections, where the total number of variables (possible
preference rankings) is $24$. We hope that further developments of these algorithms will enable the overcoming of this difficulty.
\end{quotation}
With the efficient parallelization of Normaliz in 2012 and the addition of Schürmann's symmetrization method \cite{Sch}, Normaliz could compute  a wide variety of social choice polytopes for $4$ candidates; see Bruns, Ichim and Söger \cite{BIS2}.

When the limitations of the available algorithms became visible in 2017, the author implemented an algorithm for polytope volumes by descent in the face lattice. It is implicitly based on a reverse lexicographic triangulation of the polytope, but does not compute the triangulation explicitly. It brought a significant improvement in computation times for polytopes defined by inequalities, and made more computations for elections with $4$ candidates possible, as shown in Bruns and Ichim \cite{Descent}.

But the case $n=5$, $n!=120$, remained elusive. The breakthrough came with the Normaliz implementation of Lawrence's algorithm \cite{Law}. It is based on a duality between ``generic'' triangulations of the dual cone and signed decompositions of the ``primal'' polytope into simplices. In principle, signed decompositions are as good as ordinary ones for volume computation, but they present hard numerical problems. The rational arithmetic of Normaliz can cope with them, but it must pay by computation time. The applications to $5$ candidates elections have been documented by Bruns and Ichim \cite{5cand}. They would have been unreachable without the sophisticated implementation that we explain in Section \ref{Lawrence}.

Simultaneously with the Lawrence algorithm, we introduced a refinement of the descent algorithm that identifies isomorphic faces in the descent. Isomorphism classes, as explained in \cite{nmz_auto}, are computed by McKay and Piperno's package nauty \cite{nauty}. Even for nauty, isomorphism classes are expensive, but they help in the volume computation of some classical polytopes.

Section \ref{prelim} gives a very brief introduction to the terminology of this note. It explains basic results that are used in the computation of lattice normalized volume. Section \ref{vol_alg} contains an overview of the volume algorithms and explains them, in particular Lawrence's algorithm. The final section \ref{bench} lists computation times, with emphasis on the new algorithms, signed decomposition and descent with the exploitation of isomorphism types. They are not only applied to polytopes from social choice, but also to classical polytopes, for example cubes, Birkhoff polytopes and linear ordering polytopes.

The package vinci \cite{vinci} contains algorithms for polytope volumes. See \cite{Descent} for a comparative study of computation times and memory usages. Because of its floating point arithmetic, vinci is often faster than Normaliz, but its results come without an error bound. Because of the extreme numerical difficulty, its implementation of the Lawrence algorithm fails reliable results already for polytopes coming from $4$ candidates elections; see Remark \ref{fail_vinci}.

This note is dedicated to my friends  Wolmer Vasconcelos and Rafael Villarreal. Their constant support has been very encouraging in the $25$ years of the Normaliz project. One of the first third party publications citing Normaliz is their paper \cite{Mult} with Delfino, Taylor and Weininger. The example collection of Normaliz still contains input files supplied by Rafael a quarter of a century ago, and his book \cite{Rafael} documents numerous applications.

\emph{Acknowledgement.}\enspace The author was partially supported by the DFG grant BR 688/26-1. He thanks Ulrich von der Ohe for fruitful discussions.

\section{Preliminaries}\label{prelim}

We refer the reader to \cite{BrGu} for discrete convex geometry. Here we content ourselves to a very brief overview.

\subsection{Cones and polytopes}
A \emph{cone} $C$ in the real space $\RR^n$ is the intersection of finitely many linear halfspaces:
\begin{equation}\label{half}
C=\bigcap_{i=1}^s H_i^+,
\end{equation}
and for each $i$ the halfspace $H_i^+$ is the set $\{x\in\RR^n:\sigma_i(x)\ge 0 \}$ for a linear form $\sigma_i$ in the dual space $(\RR^n)^*$. By the theorem of Minkowski-Weyl, one can equivalently describe cones as the conical set generated by finitely many vectors $v_j\in \RR^n$,
\begin{equation}\label{gen}
C=\{q_1v_1+\dots+q_nv_m: q_1,\dots,q_m\ge 0 \}.
\end{equation}
Since we want to deal only with polytopes and cones derived from them, we can restrict our cones to a subclass: $C$ is a \emph{pointed} cone: if $-x\in C$ for $x\in C$, then $x=0$. If $C$ is pointed, then the elements in a minimal set of generators as in \eqref{gen} are uniquely determined up to positive scalars, and the sets $\RR_+ v_i$ are the \emph{extreme rays} of $C$.

The cone $C$ is \emph{rational} if the vectors $v_i$ can be chosen in $\QQ^n$, and therefore in $\ZZ^n$. Then each extreme ray contains exactly one \emph{primitive} integral vector, namely one  with coprime coordinates. It is called an \emph{extreme integral generator}.

The dimension of $C$ is the dimension of the vector subspace $\RR_+ C$. If $\dim C < n$, then the halfspaces $H_i^+$ in \eqref{half} are not uniquely determined, but the halfspaces $H_i^+\cap \RR C$ of $\RR C$ in an irredundant representation $C=\bigcap _i (H_i^+\cap \RR C)$ are. They intersect $C$ in its \emph{facets}. More generally, a \emph{face} of $C$ is the intersection of $C$ with a hyperplane that has $C$ inside one of the two closed halfspaces it defines. A face of $C$ is again a cone.

In  Sections \ref{signed_dec} and \ref{Lawrence} the dual cone $C^*$ will play a central role. Its definition does not only depend on the intrinsic structure of $C$, but also on the ambient space. Therefore we will then assume that $C$ is a full dimensional pointed cone: $\dim C =n$. Consequently the halfspaces $H_i^+$ in an irredundant representation \eqref{half} are uniquely determined, and there is a unique primitive choice for $\sigma_i$. These linear forms $\sigma_i$ are called the \emph{support forms} of $C$. In this note the hyperplanes $H_i$ are the \emph{support hyperplanes} of $C$.
The \emph{dual cone} 
$$
C^*=\{\lambda\in (\RR^n)^*: \lambda(x)\ge 0 \text{ for all } x\in C\}
$$
is full dimensional and pointed as well. Under the natural identification $\RR^n = (\RR^n)^{**}$ the bidual cone $C^{**}$ is identified with $C$: the extreme rays of $C^*$ are the linear forms defining the facets of $C$, and vice versa. In the rational case the extreme integral generators of $C^*$ are the support forms of $C$, and vice versa.

A \emph{polytope} $P$ is the convex hull of finitely many points in a real space $\RR^n$. Our polytopes will be rational: such polytopes have vertices in $\QQ^n$. Computationally, polytopes are treated as compact intersections of pointed cones and hyperplanes. The hyperplane is defined by a linear form with integral coefficients, called \emph{degree}, such that
\begin{equation}\label{degree}
P = \{x\in C: \deg x = 1\}.
\end{equation}
The intersection $P$ is compact (and nonempty) if and only if $C\neq 0$ and $\deg x >0$ for $x\in C$, $x\neq 0$. This is not a restriction of generality: if $P\subset \RR^n$ is not given as in \eqref{degree}, then we can easily re-embed it suitably: we identify $P$ with $P'=P\times\{1\}\subset \RR^{n+1}$, and choose $C=\RR_+P'$.

\subsection{Lattice normalized volume}

Normaliz computes lattice normalized volume. We review this notion with emphasis of its computation. The reader can find more details in \cite[Sect. 3]{Descent}. Let $P\subset\RR^n$ be a rational polytope.
The affine hull $A=\aff P $ is a rational affine subspace of $\RR^n$. First assume that $0\in A$. Then $L=(\aff P) \cap\ZZ^n$ is a subgroup of $\ZZ^n$ of rank $d=\dim P$ (and $\ZZ^n/L$ is torsionfree). Choose a $\ZZ$-basis $v_1,\dots,v_d$ of $L$. The \emph{lattice (normalized) volume} $\Vol$ on $A$ is the Lebesgue measure on $A$ scaled in such a way that the simplex $\conv(0,v_1,\dots,v_d)$ has measure $1$. The definition is independent of the choice of $v_1,\dots,v_d$ since all invertible $d\times d$ matrices over $\ZZ$ have determinant $\pm 1$. If $0\notin A$, then we replace $A$ by a translate $A_0=A-w$, $w\in A$, and set $\Vol X =\Vol (X-w) $ for  $X\subset A$. This definition is independent of the choice of $w$ since $\Vol$ is translation invariant on $A_0$. Note that the polytope containing a single point $x\in\QQ^n$ has lattice volume $1$. If desired, the definition of lattice volume can be extended to arbitrary measurable subsets of $A$, and Normaliz does it for algebraic polytopes.

If $P$ is a lattice polytope, i.e., a polytope with vertices in $\ZZ^n$, then $\Vol P $ is an integer. For an arbitrary rational polytope we have $\Vol P \in\QQ$. As a consequence, $\Vol P $ can be computed precisely by rational arithmetic. This is not true for Euclidean volume in general: the diagonal of the unit square has length $\sqrt 2$. 

A second invariant we need is the lattice height of a \emph{rational} point $x$ over a rational subspace $H\neq\emptyset$. More generally, one can consider points $x$ such that $\aff (H,x) $ is again rational; for example, this is the case if $H$ is a hyperplane in $\RR^n$. If $x\in H$, we set $\Ht_H(x)=0$. Otherwise let $A=\aff (H,x) $ so that $H$ is a hyperplane in $A$. 

Assume first that $0\in H$. Then $H$ is cut out from $A$ by an equation $\lambda(y)=0$ with a primitive $\ZZ$-linear form $\lambda$ on $L=A\cap \ZZ^n$. With this choice of $\lambda$, $\Ht_H(x)=|\lambda(x)|$ is called the\emph{ lattice height} of $x$ over $H$. (There are exactly two choices for $\lambda$, differing by the factor $-1$.) If $0\notin H$, then we choose an auxiliary point $v\in H$, replace $H$ by $H-v$, $A$ by $A-v$ and $x$ by $x-v$. In the algorithms we will only have to deal with the case $0\in H$.
If $P$ is a rational polytope  and $F$ is a facet or, more generally, a face of $P$, then we set $\Ht_F(x)=\Ht_H(x)$ where $H=\aff F $.

The following proposition relates lattice volume and lattice height.

\begin{proposition}\label{pyramid}
Let $P$ be a rational polytope and $v\in P$ a vertex of $P$ such that there is a single facet $F$ of $P$ with $v\notin F$. Then
$$
\Vol P = \Ht_F(v)\Vol F .
$$
\end{proposition}

This is part of \cite[Prop. 1]{Descent}, to which we refer for the proof. The next basic result tells us how to compute the volume of a \emph{simplex}, which is a polytope of dimension $d$ with $d+1$ vertices.

\begin{proposition}\label{simplex}
Let $S\subset\RR^n$ be a rational simplex with vertices $v_0,\dots,v_d$. Choose a basis $u_1,\dots,u_d$ of the lattice $\aff (S-v_0) \cap \ZZ^n$. Define the $d\times d$ matrix $T=(t_{ij})$  by the representations $v_i-v_0=\sum_{j=1}^{d} t_{ij}u_j$, $i=1,\dots,d$. Then
$$
\Vol S =|\det T |.
$$
\end{proposition}

This follows immediately from the transformation formula of Lebesgue measure.  See \cite[2.C]{BrGu} for an algebraic proof.

As mentioned already, we present rational polytopes $P$ in the form $P=C\cap H$ where $C$ is a pointed cone and $H$ is defined by the condition $\deg x = 1$ with a $\ZZ$-linear form $\deg$. This brings a second polytope into play, namely $\overline P=\conv(0,P)$ as in Figure \ref{fig_ig_vonv0P}.
\begin{figure}[hbt]
\begin{center}
\begin{tikzpicture}[scale=0.7]
\filldraw[color=yellow] (-2,2) -- (0,0) --(2,2) --cycle;
\draw  (-3,3) --(0,0) -- (3,3);	
\draw[thick, color=red] (-2,2) -- (2,2);
\draw (-4,2) -- (4,2);
\draw node at (1.5,2.3){$P$};
\draw node at (0.0, 1.2){$\overline P$};
\draw node at (3.0,1.6){$\deg = 1$};
\end{tikzpicture}
\end{center}\caption{$P$ and $\overline P$}
\label{fig_ig_vonv0P}
\end{figure}

 All algorithms of Normaliz compute $\Vol \overline P$, and then derive $\Vol P$ from it:

\begin{proposition}\label{conv_0}
With the notation introduced, let $L=\RR C \cap \ZZ^n$ and $\deg|L=k\deg'$ with a primitive linear form $\deg'$ on $L$ and $k>0$. Then
$$
\Vol P = k\Vol \overline P.
$$
\end{proposition}

\begin{proof}
Let $F=P$ be the unique facet of $\overline P$ opposite to $0$. 
We can use $\deg'$ to measure lattice height over $F$. Since $\deg'x= (1/k)\deg x=1/k$ for $x\in F$, one has $\Ht_H(0)=1/k$, and the claim follows from Proposition \ref{pyramid}.
\end{proof}

The number $k$ in Proposition \ref{conv_0} is called the \emph{grading denominator} in Normaliz. The reason is that $\deg'=\deg/k$ is considered as the ``true'' grading on the cone $C$. The user can choose between the given grading $\deg$ or the divided one, $\deg'$.

As our final tool we formulate a homogeneous version of Proposition \ref{simplex}:

\begin{proposition}
Let the simplex $S$ be given in the form $S=C\cap H$ where $H$ is the hyperplane of degree $1$ points and $C= \RR_+S$. Let $v_1,\dots,v_d$, $d=\dim S+1$, be nonzero points in the $d$ extreme rays of $C$, for example the extreme integral generators. Then
$$
\Vol \overline S= \frac{1}{g_1\cdots g_d}|\det T|,\qquad g_i=\deg v_i, \ i=1,\dots,d,
$$
where $T=(t_{ij})$ is the $d\times d$ matrix with $v_i =\sum_j t_{ij} u_j$ for a basis $u_1,\dots,u_d$ of the lattice $L=\ZZ^n\cap \RR S$.
\end{proposition}

This follows immediately from Proposition \ref{simplex} if we set $v_0 = 0$, observing that $v_1/g_1,\dots,\allowbreak v_d/g_d$ are the remaining vertices of $\overline S$.

\section{Volume algorithms in Normaliz}\label{vol_alg}

There are three basic algorithms:
\begin{enumerate}
\item the \emph{primal} volume algorithm: Normaliz computes a lexicographic triangulation, and finds the volume as the sum of the volumes of the simplices in the triangulation;
\item volume by \emph{descent in the face lattice}: there is a reverse lexicographic triangulation  in the background, but it is not computed explicitly;
\item volume by \emph{signed decomposition}, the \emph{Lawrence algorithm}: Normaliz computes a triangulation of the dual cone and converts it into a signed decomposition of the polytope.
\end{enumerate}

Normaliz also computes the exact volume of full dimensional polytopes defined over real algebraic number fields. For them  only (1) is implemented at present. One could extend (3) to them, whereas (2) is not suitable. The algorithms (1) and (3) are also used in the computations of integrals of rational polynomials over polytopes.

By rule of thumb one can say that  the best choice is
\begin{enumerate}
	\item if the polytope has few vertices, but potentially many facets;
	\item  if the number of vertices and the number of facets are of the same order of magnitude;
	\item  if there are \emph{very} few facets and many vertices.
\end{enumerate}
This recommendation will be confirmed by the computational data in Section \ref{bench}. There are variants:

\begin{enumerate}
	\item[(a)] exploitation of isomorphism types of faces in the descent algorithm;
	\item[(b)] symmetrization as explained below.
\end{enumerate}

Normaliz checks the default conditions of the algorithms in the order
\begin{center}
	signed decomposition $\to$  descent  $\to$ symmetrization.
\end{center}
If the default conditions are not satisfied for any of them, the primal triangulation algorithm is used. These decisions must often be made on the basis of partial information. Therefore it can be useful to choose a certain variant explicitly or to exclude others. The exploitation of isomorphism types must always be asked for by the user. 

Normaliz uses OpenMP for parallelization. Unless the user insists on computations with GMP integers, Normaliz tries $64$ bit arithmetic first, and restarts the computation with GMP integers if it recognizes an overflow.

\subsection{The primal volume algorithm}

Mathematically there is not much to say: if a polytope $P$ is decomposed into simplices with non-overlapping interiors, then its volume is the sum of the volumes of the simplices forming the decomposition.
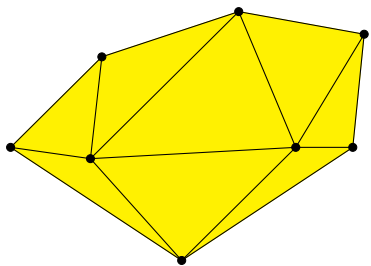
\begin{figure}[hbt]
\begin{center}
	\begin{tikzpicture}[scale=1.5]
	\filldraw[color=yellow] (0,0) -- (1.5,1) -- (1.6,2) -- (0.5,2.2) -- (-0.7,1.8) -- (-1.5,1) -- cycle;
	\draw (0,0) -- (1.5,1) -- (1.6,2) -- (0.5,2.2) -- (-0.7,1.8) -- (-1.5,1) -- cycle;
	\filldraw[fill=black] (0,0)  circle (1pt);
	\filldraw[fill=black] (1.5,1)  circle (1pt);
	\filldraw[fill=black] (1.6,2)  circle (1pt);
	\filldraw[fill=black] (0.5,2.2)  circle (1pt);
	\filldraw[fill=black] (-0.7,1.8)  circle (1pt);
	\filldraw[fill=black] (-1.5,1)  circle (1pt);
	\filldraw[fill=black] (1,1)  circle (1pt);
	\filldraw[fill=black] (-0.8,0.9)  circle (1pt);
	\draw (1,1) --  (-0.8,0.9);
	\draw (-1.5,1) --  (-0.8,0.9);
	\draw (0,0) --  (-0.8,0.9);
	\draw (1,1) --  (0,0);
	\draw (1,1) --  (1.6,2);
	\draw (1,1) --  (1.5,1);
	\draw (1,1) --  ((0.5,2.2);
	\draw (-0.8,0.9) --  (0.5,2.2);
	\draw (-0.8,0.9) --  (-0.7,1.8);
	\end{tikzpicture}
\end{center}\caption{A triangulation}
\label{fig_triang}
\end{figure}
Since the computation of Hilbert bases and Hilbert series is based on (lexicographic) triangulations as well, Normaliz has a sophisticated  algorithm for them, using pyramid decomposition; see \cite{BIS}. Normaliz tries to avoid determinant computations by the ``exploitation of unimodularity''; see \cite[Prop. 7]{BIS}.

\subsection{Volume by descent in the face lattice}

The idea is to exploit the following proposition:

\begin{proposition}\label{decomp}
Let $P\subset\RR^n$ be a rational polytope, and $v\in P$. Then
\begin{equation}
\Vol P =\sum_{F\text{ facet of }P} \Ht_F(v)\Vol F .\label{Vol_dec}
\end{equation}
\end{proposition}

Proposition \ref{decomp} follows immediately from Proposition \ref{pyramid} since the polytopes $\conv(v,F)$ constitute a polyhedral decomposition of $P$.  Usually $v$ is a vertex of the polytope $P$ with as few opposite facets $F_i$ as possible, as illustrated by Figure \ref{fig_decomp}.
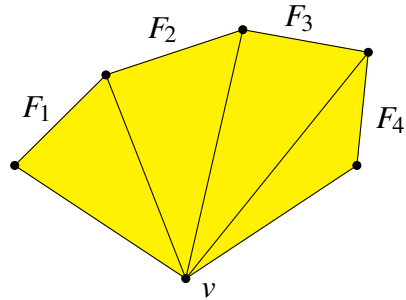
\begin{figure}[hbt]
\begin{center}
	\begin{tikzpicture}[scale=1.5]
	\filldraw[color=yellow] (0,0) -- (1.5,1) -- (1.6,2) -- (0.5,2.2) -- (-0.7,1.8) -- (-1.5,1) -- cycle;
	\draw (0,0) -- (1.5,1) -- (1.6,2) -- (0.5,2.2) -- (-0.7,1.8) -- (-1.5,1) -- cycle;
	\draw (0,0) -- (0.5,2.2);
	\draw (0,0) -- (1.6,2);
	\draw (0,0) -- (-0.7,1.8);
	\draw node at (0.2,-0.1){$v$};
	\filldraw[fill=black] (0,0)  circle (1pt);
	\filldraw[fill=black] (1.5,1)  circle (1pt);
	\filldraw[fill=black] (1.6,2)  circle (1pt);
	\filldraw[fill=black] (0.5,2.2)  circle (1pt);
	\filldraw[fill=black] (-0.7,1.8)  circle (1pt);
	\filldraw[fill=black] (-1.5,1)  circle (1pt);
	\draw node at (-1.3,1.5){$F_1$};
	\draw node at (-0.2,2.2){$F_2$};
	\draw node at (1.0,2.3){$F_3$};
	\draw node at (1.8,1.4){$F_4$};
	\end{tikzpicture}
\end{center}\caption{Pyramid decomposition of a polytope}\label{fig_decomp}
\end{figure}
The recursive application results  in building a \emph{descent system}, i.e., a subset $\mathcal F$ of the face lattice so that for each face $F\in\mathcal F$, to which \eqref{Vol_dec} is applied, all facets of $F$ that are opposite to the selected vertex are contained in $\mathcal F$. However, if a face is simplicial, its multiplicity is computed by the standard determinant formula. The algorithm is implemented in such a way that all data are collected in the descent and no backtracking is necessary. The RAM usage is essentially determined by the two largest layers. For a detailed discussion we refer the reader to \cite{Descent}.

\subsection{Exploitation of isomorphism classes}

If the integral automorphism group of the cone $C$ over the polytope $P$ is large enough, one can expect that each face in the descent system $\cF$ has many isomorphic copies in $\cF$. These can be detected and identified so that only one representative of every isomorphism class must be kept in $\cF$. This reduces $\cF$ in size and can significantly speed up the volume computation. It must be used with care since the computation of isomorphism classes is rather expensive. See \cite{nmz_auto} for a discussion of their computation.

If the polytope is specified by generators and the number of facets is large, then the first step in the descent system is built differently. Normaliz computes the automorphism group of the polytope and selects one representative in each orbit of facets. The vertex $v$ above is replaced by a fix point of the automorphism group, and the first step in the volume computation is the formula
$$
\Vol P =\sum_{i=1}^c O(F_i)\Ht_{F_i}(v) \Vol F_i
$$
where $F_i$ represents one of the $c$ orbits and $O(F_i)$ is the number of facets in the orbit of $F_i$. Then $F_1,\dots,F_c$ form the first layer in the descent system. This allows the application of descent in cases where the number of facets is too large for a successful computation without exploitation of isomorphism classes. If the user does not prohibit it, Normaliz encodes isomorphism classes by their SHA256 checksums.

\subsection{Symmetrization}
To understand the computation of volumes through symmetrization one must take a detour through Ehrhart series. As usual, assume that our polytope $P$ is given as the intersection $P=C\cap H$ where $C\subset \RR^d$ is a pointed rational cone and $H=\{x\in  \RR^d: \deg x=1\}$ is the hyperplane of degree $1$ points. For symmetrization we assume that $\deg$ is primitive.

Under certain conditions one can count lattice points of degree $k$, $k\in \NN$, in $C$ by mapping $C$ to a cone $C'$ of lower dimension and then counting each degree $k$ lattice point $y$ in $C'$ with the number of its lattice preimages. This approach works well if the number of preimages is given by a polynomial in the coordinates of $y$. Since $C'$ has lower dimension, one can hope that its combinatorial structure is much simpler than that of $C$. One must of course pay a price: instead of counting each lattice point with the weight $1$, one must count it with a polynomial weight.

The availability of this approach depends on symmetries in the coordinates of $C$, and therefore we call it \emph{symmetrization}. Normaliz tries symmetrization under the following condition: $C$ and the relevant lattice are given by constraints (inequalities, equations, congruences) and the inequalities contain the sign conditions $x_i\ge 0$ for all coordinates $x_i$ of $C$. Then Normaliz groups coordinates that appear in all constraints and the grading~(!) with the same coefficients, and, roughly speaking, replaces them by their sum. The number of preimages that one must count for the vector $y$ of sums is then a product of binomial coefficients---a polynomial as desired. More precisely, if $y_j$, $j=1,\dots,m$, is the sum of $u_j$ variables $x_i$ then
$$
f(y)=\binom{u_1+y_1-1}{u_1-1}\cdots \binom{u_m+y_m-1}{u_m-1}.
$$
is the number of preimages of $(y_1,\dots,y_m)$.

Since the Lebesgue measure can be approximated by scaled counting measures, one obtains
$$
\Vol P=\int_P  h\,d\negthinspace\lambda
$$
where $h$ is the highest homogeneous component of $f$ with respect to total degree, and $\lambda$ is the suitably scaled Lebesgue measure.  We learnt this approach from Sch\"urmann~\cite{Sch}. The Normaliz algorithm for integrals is described in \cite{BS}. This note contains a complete elementary treatment and several references to advanced aspects.

Symmetrization can have stunning effects. Nevertheless we do not include it in the computations of Section \ref{bench} since it does not help for any of them, at least not in the present implementation. Plenty of examples are contained in \cite{BIS2}, where it is often very useful in the computation of Hilbert series. 

\subsection{Volume by signed decomposition}\label{signed_dec}

This algorithm uses that a generic triangulation of the dual cone induces a signed decomposition of the primal polytope, as we will now explain.

Let $C\subset \RR^d$ be a pointed cone of dimension $d$ (it is important that $C$ is full dimensional). The polytope $P$ is the intersection of $C$ with the hyperplane $H$ defined by a grading $\deg$: $H=\{x:\deg(x)=1\}$. The grading is an interior element of the dual cone $C^*=\{\lambda\in(\RR^d)^*:\lambda(x)\ge 0 \text { for all }x\in C  \}$. In order to visualize the situation we take an auxiliary (irrelevant) cross-section $Q$ of the dual cone as in Figure \ref{fig_square}.
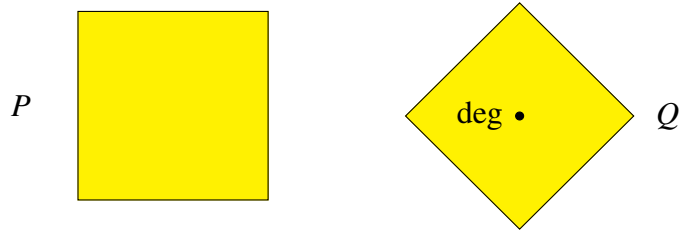
\begin{figure}[hbt]
\begin{center}
	\begin{tikzpicture}[scale=2.5]
	\filldraw[color=yellow] (0,0) -- (1,0) -- (1,1) -- (0,1) -- cycle;
	\draw (0,0) -- (1,0) -- (1,1) -- (0,1) -- cycle;
	\draw node at (-0.3,0.5){$P$};
	\draw node at (0,-0.1){};
	\end{tikzpicture}
	\qquad\qquad
	\begin{tikzpicture}[scale=1.5]
	\filldraw[color=yellow] (1,0) -- (0,1) -- (-1,0) -- (0,-1) -- cycle;
	\draw (1,0) -- (0,1) -- (-1,0) -- (0,-1) -- cycle;
	\draw node at (1.3,0){$Q$};
	\filldraw (0,0) circle (1pt);
	\draw node at (-0.35,0){$\deg$};
	\end{tikzpicture}
\end{center}\caption{A square and a cross-section of the dual cone}\label{fig_square}
\end{figure}

Now suppose that we have a \emph{generic} triangulation $\Delta$ of the dual cone where genericity is defined as follows: $\deg$ is not contained in any hyperplane through a facet of any $\delta\in\Delta$. Let $\delta\in\Delta$ be given, and denote the linear forms on $(\RR^d)^*$ defining its facets by $\ell_1,\dots\ell_d\in (\RR^d)^{**} = \RR^d$. ( $\ell_1,\dots\ell_d$ are the extreme rays of the dual of $\delta$.) The hyperplanes defined by the vanishing of $\ell_1,\dots\ell_d$ decompose $(\RR^d)^*$ into ``orthants'' that can be labeled by a sign vector $\sigma=(s_1,\dots,s_d)\in \{ \pm 1  \}^d$:
$$
D(\delta,\sigma)=\{\alpha: (-1)^{s_i} \ell_i(\alpha) \ge 0   \}.
$$
By the assumption on $\deg$, there is \emph{exactly on}e sign vector $\sigma$ such that $\deg $  lies in the interior of $ D(\delta,\sigma)$. Consequently the hyperplane $H$ intersects the dual  $D(\delta,\sigma)^*$ in a polytope $R_\delta$. Set $e(\delta)=|\{i: s_i=-1 \}|$.

Let $\iota_X$ denote the indicator function of a subset $X\subset \RR^d$. Then
\begin{equation}
\iota_P(x) = \sum_{\delta\in \Delta} (-1) ^{e(\delta)} \iota_{R_\delta}(x)\label{iota}
\end{equation}
for all $x\in\RR^d$ outside a union of finitely many hyperplanes. Since volume (lattice normalized or Euclidean) is additive on indicator functions, this formula can be used for the computation $\Vol P$, and more generally for the computation of integrals over~$P$.

In order to find a generic triangulation, Normaliz first computes a triangulation $\Delta$ of $C^*$ and saves the induced \emph{hollow triangulation} $\Gamma$ that $\Delta$ induces on the boundary of $C^*$. Then it finds a \emph{generic} element $\omega\in C^*$ such that the \emph{star triangulation} $\Sigma$ of $C^*$, in which every simplicial cone is generated by the \emph{center} $\omega$  and a facet of the hollow triangulation, is generic. Figure \ref{fig_lawrence} illustrates the signed decomposition of a square into $4$ simplices.
\begin{figure}[hbt]
\begin{center}
	\begin{tikzpicture}[scale=1.5]
	\filldraw[color=yellow] (0,0) -- (1,0) -- (1,1) -- (0,1) -- cycle;
	\draw (0,0) -- (1,0) -- (1,1) -- (0,1) -- cycle;
	\draw node at (0.5,0.5){$P$};
	\draw node at (0,-0.1){};
	\draw (-1.3,1) -- (1,1) -- (1,-1.3) -- cycle;
	\draw (1,0) -- (-0.3,0);
	\draw (0,1) -- (0,-0.3);
	\draw node at (0.3,0.3){$+$};
	\draw node at (-0.5,0.5){$-$};
	\draw node at (0.45,-0.45){$-$};
	\draw node at (-0.2,-0.2){$+$};
	\end{tikzpicture}
	\qquad\qquad
	\begin{tikzpicture}[scale=1.5]
	\filldraw[color=yellow] (1,0) -- (0,1) -- (-1,0) -- (0,-1) -- cycle;
	\draw (1,0) -- (0,1) -- (-1,0) -- (0,-1) -- cycle;
	\draw node at (1.3,0){$Q$};
	\filldraw (0,0) circle (1pt);
	\draw node at (-0.3,-0.1){$\deg$};
	\filldraw (0.3,0.3) circle (1pt);
	\draw (-1,0) -- (0.3,0.3) -- (0,1);
	\filldraw (0.3,0.3) circle (1pt);
	\draw (1,0) -- (0.3,0.3) ;
	\draw (0,-1) -- (0.3,0.3) ;
	\draw (0,-1) -- (0.3,0.3) ;
	\draw node at (0.45,0.05){$\omega$};
	\draw node at (-0.3,-0.45){$+$}; 
	\draw node at (0.5,0.4){$+$};
	\draw node at (-0.3,0.5){$-$};  
	\draw node at (0.5,-0.3){$-$};	
	\end{tikzpicture}
\end{center}\caption{Generic triangulation of the dual and signed decomposition}
\label{fig_lawrence}
\end{figure}
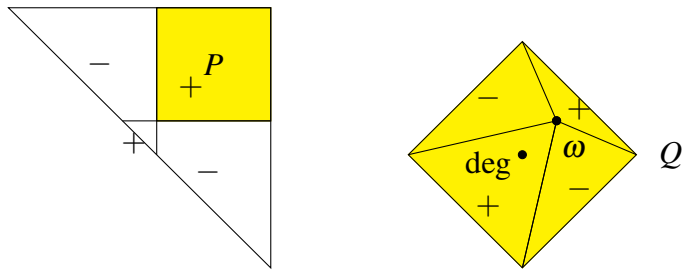 

The algorithm has been developed by Lawrence \cite{Law} in the language of linear programming. We have learnt it from Filliman's paper \cite{Fil}, which contains a proof of equation \eqref{iota}. See  Beck, Haase and Sottile \cite{BHS} for an elementary approach and the relationship to other duality theorems.

\section{The implementation of the Lawrence algorithm}\label{Lawrence}

The complete course of the volume computation consists of $4$ steps that can be clearly delimited from each other:
\begin{enumerate}
	\item computation of a triangulation $\Delta$ of $C^*$;
	\item computation of the induced hollow triangulation $\Gamma$;
	\item choice of the generic element $\omega$;
	\item evaluation of the star triangulation $\Sigma$.
\end{enumerate}
Each of the $4$ steps is highly critical if one wants to reach the applications in social choice that were our driving challenge. For (1) we could essentially rely on the standard triangulation algorithm of Normaliz. Step (2) and the arithmetic for (3) and (4) are described in the following. Both (3) and (4) are iterations over the hollow triangulation and star triangulations derived from it.

\subsection{The hollow triangulation} 
Suppose the triangulation $\Delta$ of $C^*$ has been computed. For each simplicial cone $\delta\in\Delta$ we must now find the facets of $\delta$ that lie in the boundary of $C^*$. There are various solutions for this task. The first that comes to mind is to compute the facets of $C^*$ and match the facets of $\delta$ with it. But $C^*$ can have an enormous number of facets that one does better not compute since they can easily exhaust RAM. The facets of $C^*$ are of course extreme rays of the cone $C$ over $P$, but for signed decomposition Normaliz only computes them if asked for by the user. A second approach that is much better in terms of RAM is to compute the facets of $\delta$ and select those that have all extreme rays of $C^*$ on the same side as $\delta$. However, this requires an enormous number of scalar products that in high dimension are expensive. 

Instead of these geometric approaches,  Normaliz uses only the combinatorics of triangulations of cones. It forms the set of all pairs $(F,\delta)$ where $F$ is a facet of $\delta$, $\delta\in \Delta$. Both components are encoded as $0$-$1$-vectors that indicate the extreme rays of $C^*$ spanning $F$ and $\delta$, respectively. From this set one must discard all pairs $(F,\delta)$ for which there exists a pair $(F,\delta')$, $\delta'\in\Delta$, $\delta'\neq \delta$. In  principle one could eliminate all $F$ that appear a second time without remembering the ``mother'' $\delta$, but $\delta$ helps in several ways. The first is that one can store the hollow triangulation as  a set of pairs $(\delta,\phi(\delta))$ where $\phi(\delta)$ is a second $0$-$1$-vector indicating those extreme rays of $\delta$ whose omission yields a facet of the hollow triangulation.

In order not to blow up memory for large $\Delta$, the pairs $(\delta,\phi(\delta))$ are computed in small portions controlled by ``patterns''. Each pattern is an increasing sequence $(p_1,\dots,p_r)$ of indices, and it is required that the facet $F$ satisfies the following condition: if $q_1,\dots, q_{d-1}$, $d=\dim C^*$, are the indices of the extreme rays of $F$ in ascending order, then $q_i=p_i$ for $i=1,\dots,r$.

\subsection{Piggyback simplices}\label{piggy}
After the purely combinatorial computation of the hollow triangulation, arithmetic must be used in steps (3) and (4) above, namely in finding a generic linear form $\omega\in C^*$ and then in the volume computation. Both tasks are accelerated significantly if one takes advantage of the fact that simplices $G$ and $G'$ of the star triangulation are in ``piggyback'' position to each other, if the facets $F$ and $F'$ of the hollow triangulation that define them belong to the same simplex $\delta\in\Delta$. By ``piggyback'' position we mean that the simplices share a facet and lie on different sides of it, as indicated in Figure~\ref{fig_piggy}.
\begin{figure}[hbt]
\begin{center}
\begin{tikzpicture}[scale = 1.0]
\filldraw[color=yellow] (-2.5,1) -- (0,2.5) -- (1.5,1.7) -- (0,0) --cycle;
\draw (-2.5,1) -- (0,2.5) -- (1.5,1.7) -- (0,0) --cycle;
\draw (-2.5,1) -- (1.5,1.7);
\draw node at (1.0,2.3){$F$};
\draw node at (1.0,0.5){$F'$};
\draw node at (-0.3,1.7){$G$};
\draw node at (-0.3,0.8){$G'$};
\draw node at (-2.8,1.1){$\lambda_1$};
\draw node at (0,2.7){$\lambda_3$};
\draw node at (1.7,1.7){$\lambda_2$};
\draw node at (0,-0.3){$\mu$};
\end{tikzpicture}	
\end{center}\caption{Piggyback simplices}
\label{fig_piggy}	
\end{figure}
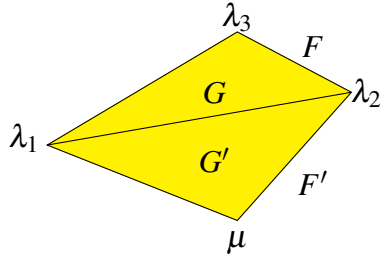

Suppose that $G$ has extreme rays $\lambda_1,\dots,\lambda_d\in (\RR^d)^*$ and its facets are given by linear forms $\ell_1,\dots,\ell_d\in \RR^d = (\RR^d)^{**}$ where $\ell_d$ defines the common facet, $\lambda_1,\dots,\lambda_{d-1}$ are the extreme rays shared by $G$ and $G'$ and $\mu$ is the remaining extreme ray of $G'$. Then the facets of $G'$ are determined by $-\ell_d$ and
\begin{equation}
k_i = -\ell_d(\mu)\ell_i + \ell_i(\mu)\ell_d, \qquad i=1,\dots,d-1. \label{FM}
\end{equation}
Since the computation of the $k_i$ and $-\ell_d$ from $\lambda_1,\dots,\lambda_{d-1},\mu$ alone amounts to the inversion of a matrix, it is clear that the use of \eqref{FM} is a significant advantage, even if the computation of the values $\ell_i(\mu)$ needs $d^2$ multiplications.There is actually no need to compute the $k_i$ completely. We will only need their values on elements in the dual space, for which the values of the $\ell_i$ are known, for example the degree. 

In the primal space, $\ell_1,\dots,\ell_{d}$ and $k_1,\dots,k_{d-1},-\ell_d$ are extreme rays of the simplices dual to $G$ and $G'$. For the volume of the corresponding simplices we need the determinants. By standard rules
$$
|\det (k_1,\dots,k_{d-1},-\ell_d) | = |\ell_d(\mu)^{d-1}\det (\ell_1,\dots,\ell_{d})|.
$$
So the piggyback relation between $G$ and $G'$ pays off a second time.

In dealing with the simplices of the star triangulation that belong to the same simplicial cone $\delta$ of $\Delta$, we pick one of them, say $G_1$, and take all others piggyback. For $G_1$ we must indeed invert the matrix $M$ with rows $\lambda_1,\dots,\lambda_d$ over $\QQ$, using $D=|\det M|$ as the denominator: $M^{-1}= (1/D)N$ with a matrix $N\in \ZZ^{d\times d}$. After extraction of their greatest common divisors, the columns $\ell_1,\dots,\ell_d$ of $N$ are the support forms of $G_1$, equivalently, the extreme integral generators of the dual cone of $G_1$. For the volume computation we need the determinant of the matrix $N'$ with columns $\ell_1,\dots,\ell_d$. There is no need to compute it directly: since $MN'$ is the diagonal matrix with entries $\ell_i(\lambda_i)$, $i=1,\dots,d$, one has
$$
\det N' =\frac{\prod_{i=1}^d \ell_i(\lambda_i)}{\det M},
$$
and $\det M$ has already been computed.

\subsection{Finding a generic element}
The generic element $\omega$ of $C^*$ must satisfy the following condition: the grading $\deg$ does not lie on any hyperplane through a facet of a simplicial cone in the hollow triangulation and  $\omega$. But this condition is symmetric in $\omega$ and $\deg$! It is much better to first take the star triangulation with center $\deg$ and check that $\omega$ is not on any of the critical hyperplanes. The main difference between $\deg$ and $\omega$ is the size of the coordinates: those of $\deg$ are usually very small and those of $\omega$ very large. So, in working with $\deg$ as the center, there is a very good chance to get away with $64$ bit arithmetic. The computation based on $\omega$ which is necessary for the volume, must very often be done with GMP integers.

Instead of choosing one vector $\omega$ at random and verifying that it is generic, Normaliz takes two vectors $\omega_1$ and $\omega_2$ and checks that not both of them lie together on a critical hyperplane. If this condition is satisfied, then a suitable linear combination $\omega=a_1\omega_1+a_2\omega_2$, $a_1,a_2\in\ZZ$, $a_1,a_2\ge 0$, is generic. For the check we apply the piggyback trick of Section \ref{piggy}, and especially  \eqref{FM}. If $\omega_1$ and $\omega_2$ do not work, then their coordinates are increased.

\subsection{The addition of fractions}
There is one more extremely critical aspect, namely the addition of fractions. If one accumulates the volume as an alternating sum of simplex volumes by successive addition of fractions, one can easily spend $99\%$ of the computation time on this addition, or, in extreme situations, not finish at all, creating fractions whose numerators and denominators fill gigabytes. The reason is that the extreme integral generators of the duals of cones involving the generic element can be very large: we must divide by them.

In extreme cases there is no other choice but to work with fixed precision. If the user asks for it, then the simplex volumes are still computed exactly, but for addition they are truncated  to a fixed number of decimal digits. The default choice is truncation to an integral multiple of $10^{-100}$. With this choice the final volume is computed up to an error $\le |\Gamma|\cdot 10^{-100}$. The user can set a higher or lower precision.

For computations with full precision, Normaliz uses an \emph{addition pyramid}, an extensible vector $(q_0,\dots,q_n)$ of rational numbers. It serves as an accumulator. A new summand $s$ is added to $q_0$, provided the number of summands that have already been accumulated in $q_0$ is smaller than the \emph{capacity}. If the capacity has been reached, then $q_0$ is added to $q_1$ and set to $s$. The addition of $q_0$ to $q_1$ is handled in the same way as that of $s$ to $q_0$ etc. At present the capacity is $8$. This scheme has proved to be very efficient. Of course, at the very end, all entries of the pyramid must be added. 

\begin{remark}\label{fail_vinci}
vinci \cite{vinci} contains the Lawrence algorithm in a floating point implementation. As the authors state in \cite{practical},  it is numerically unstable, and our preceding discussion should also indicate this problem. 

If one has a look at the volumes whose alternating sum must be formed, then they easily reach absolute values of $10^{100}$. For any precision of their alternating sum, which may be of order $10^{-6}$, one therefore needs a very high number of significant digits that a standard floating point format does not offer. 

For example, this becomes visible already in the comparison of $4$ voting schemes for $4$ candidates  \cite[Sect.\ 6.1]{Descent}. Let $P$ be the corresponding polytope. With its algorithm HOT (based on the same principle as descent) vinci correctly computes the Euclidean volume of $1.260510232743\cdot 10^{-25}$, for $\overline P$, whereas the vinci Lawrence algorithm yields  $9.287423132835\cdot 10^{-8}$. (We are grateful to Bogdan Ichim for these computations.) For this reason we are not comparing the Lawrence algorithms in Normaliz and vinci.
\end{remark}   

\section{Computational data}\label{bench}

 All computations have been done on the Dell R640 server of the Institute of Mathematics at Osnabrück. It is equipped with two Intel Xeon Gold 6152 cards (a total of $44$ cores) and $1$ TB of RAM. The computations use $3 2$ parallel threads (of the maximum of $88$). The listed times are ``wall clock'' times. In order to avoid overloading the tables, information about RAM usage has been inserted into the text.
 
That we allow $32$ threads for a computation does of course not mean that they can be used. The percentage of CPU that the computations got varies from $\sim 400\%$ to almost $3200\%$.

\subsection{Polytopes defined by vertices}

The computation times for several polytopes defined by vertices are listed in Table \ref{tab_vert}. In the tables, $\dim$ is the dimension of the cone $C$ over the polytope, \#ext the number of extreme rays of $C$ and $\sup$ the number of its support hyperplanes. The size of the triangulation computed by the primal triangulation is to be found in the column \#tri. The number of determinants computed is usually smaller.

\begin{table}[hbt]
\begin{tabular}{rrrrrrrrr}
\strut & \multicolumn{4}{c}{combinatorial data}&\qquad & \multicolumn{3}{c}{computation times}   \\
\cline{2-5} \cline{7-9}
\strut & dim &\#ext & \#supp &\#tri  && primal & descent & isotypes\\
\midrule[1.2pt]
\strut\ttt{lo-6} & 16 & 720&  910 & $5.8\cdot 10^9$   && 19:20.80 m & 3:17.48 m & 0:04.57 m\\	
\strut\ttt{lo-7} & 22 & 6040 & 87,472 &&&&&21:39:51 h\\
\midrule[1.2pt]
\strut \ttt{cr-20}& 21& 40& $2^{20}$ & $2^{19}$ && 0:08.50 m&0:08.94 m& 0:15.02 m\\
\strut \ttt{cr-24}& 25& 48& $2^{24}$ & $2^{23}$ && 2:11.91 m & 3:29.67 m& 5:42.24 m \\
\strut \ttt{cr-28}& 29& 48& $2^{28}$ & $2^{27}$ && 42:12.11 m& 1:39:37 h& 2:21:09 h\\
\midrule[1.2pt]
\strut \ttt{A543} & 36 & 60 & 29,387 &$103\cdot 10^6$ & & 0:24.09 m& 36:56.56 m& 0:18.59 m  \\
\strut \ttt{A553} & 43 & 75 & 306,955 & $9.2\cdot 10^9$  && 44:53.26 m&& 7:10.36 m\\
\strut \ttt{cy-60}& 17 & 60 & 656,100 & & & 0:46.45 m&& 0:44.27 m \\
\midrule[1.2pt]
\end{tabular}
\vspace*{1ex}\caption{Polytopes defined by vertices}\label{tab_vert}
\end{table}

\subsubsection{Linear ordering polytopes}

\verb|Lo<n>| is the linear ordering polytope for a set of $n$ elements. 
These polytopes have been investigated in combinatorial optimization; see \cite{MaRe}. The maximum $n$ reachable is $7$. For $n=8$ not even the number of facets is known. It is however $> 800\cdot 10^6$. The computation of the volumes is surprisingly fast if one exploits the isomorphism classes of faces. Note that for $n=7$ the computation of the support hyperplanes alone takes $ > 20$ h so that the computation of the volume needs $\sim 1$~h.

The maximum RAM usage of \verb|lo-6| is $2.1$ GB for descent, the other two algorithms need $< 1.5$ GB. The computation for \verb|lo-7|  takes $14.7$ GB.

\subsubsection{Cross polytopes}

\verb|cr-<n>| is the unit cross polytope of dimension $n$. We have computed their (known) volumes for $n=20,24,28$. They have only $2n$ vertices, but $2^n$ facets. But all facets are simplices, and therefore the descent algorithm and its variant exploiting isomorphism classes are applicable. However, the primal algorithm behaves better for two reasons: (i) it avoids the administrative overhead of the descent algorithm, and (ii) the formation of the single orbit of facets takes rather long---it cannot be parallelized. So the saving in the computations of determinants is overcompensated. Since all facets of a cross polytope are simplicial, there is only one descent step, namely from the full polytope to the facets opposite to the chosen vertex.

For $n=20$ the primal algorithm gets away with $713$ MB, whereas the two descent algorithms need about $1$ GB. For $n=24$ the numbers are $4.2$ GB and $21$ GB. For $n=28$ they rise to $194$ GB and $\sim 300$ GB. It takes a lot of space to accommodate the $2^{28}$ extreme rays.

\subsubsection{Other polytopes}
\verb|A543| and \verb|A553| are taken from the Ohsugi-Hibi classification \cite{OH} of polytopes related to contingency tables. \verb|A553| shows that descent with isomorphism types can be favorable if the automorphism group of the polytope is sufficiently large. For \verb|A543| this effect is already visible, but still small. This applies to \verb|cy-60| as well, the cyclotomic polytope of order $60$ defined by Beck and Ho\c sten \cite{BeH}. That the pure descent algorithm is not suitable for this type of polytope is shown by \verb|A543|.

The RAM usage of \verb|A543| is about $1.6$ GB for the primal algorithm and $1$ GB for descent with isomorphism classes. For \verb|A553 | the corresponding numbers are $4.3$ GB and $101$ GB. For \verb|cy-60| they are $316$ MB and $1.4$ GB.

\subsection{Polytopes defined by inequalities}

We now turn to polytopes defined by inequalities and equations. Among them we have chosen Birkhoff polytopes, cubes and polytopes from social choice---as said already, the latter were the driving challenge for our implementation of the Lawrence algorithm. In the tables, \#tri dual is the size of the triangulation of the dual cone, and \#hollow that of the associated hollow triangulation.

\subsubsection{Birkhoff polytopes}

The Birkhoff polytope of order $n$ is the set of doubly stochastic $n\times n$ matrices. Its vertices are the $n\times n$ permutation matrices, and their number $n!$ is rapidly growing. Their volumes have been computed for $n\le 10$ by Beck and Pixton \cite{BP} with residue methods that are not (yet) available in Normaliz. 

For $n\le 5$ any of the Normaliz algorithm does the job very quickly, but 
for $n=6$ the primal algorithm must already give up since the lexicographic triangulation becomes too large. In the table we start with this case. The Lawrence algorithm reaches $n=8$. As one can see, even the triangulations of the dual cone grow too quickly for the next step. The bulk of the computation time for $n=8$ is taken by the computation of the hollow triangulation, namely $\sim 13$ h. The coordinates of the generic element $\omega$ are small enough to allow 64 bit arithmetic for the volume computation ($< 3$ h).
\begin{table}[hbt]
\begin{tabular}{rrrrrrrrr}
\strut & \multicolumn{4}{c}{combinatorial data}&\qquad & \multicolumn{3}{c}{computation times}   \\
\cline{2-6} \cline{8-9}
\strut & dim &\#ext & \#supp &\#tri dual & \#hollow  &&  isotypes & signed dec\\
\midrule[1.2pt]
\strut \verb|bi-6| & 26 &  720& 36&  142,755 & 933,120 && 0:03.64 m &  0:03.10 m\\
\strut \verb|bi-7| & 37 & 5040& 49& $11\cdot 10^6$ & $85\cdot 10^6$ &&38:05.85 m &  5:36.92 m \\
\strut \verb|bi-8| &  50 & 40,320 & 64& $1.2\cdot 10^9$ & $11\cdot 10^9$ && &  17:59:20 h\\
\midrule[1.2pt]
\strut \verb|cu-20|& 21& $2^{20}$ & 40& $2^{19}$ & $2^{20}$&& 0:08.73 m & 0:07.00 m \\
\strut \verb|cu-24|& 25& $2^{24}$ & 48& $2^{23}$ & $2^{24}$&& 4:56.11 m&2:12.64 m \\
\strut \verb|cu-28|& 29& $2^{28}$ & 56& $2^{27}$ & $2^{28}$&& 1:53:27 h & 1:01:37 h\\
\midrule[1.2pt]
\end{tabular}
\vspace*{1ex}\caption{Birkhoff polytopes and cubes}
\label{tab_bi_ciube}
\end{table}

RAM usage for \verb|bi-6| is $44$  MB and $473$ MB. The computations for \verb|bi-7| take $7.9$ GB and $12.5$ GB. That for \verb|bi-8| needs $216$ GB.

\subsubsection{Cubes}
\verb|cu-<n>| is the unit cube of dimension $n$. It is a good test object since its volume is known. Since faces of the same codimension are isomorphic, descent with isomorphism types is expected to be fast, and it is indeed. However, it must use the huge number of vertices explicitly, and for this reason signed decomposition is even faster. The Normaliz binary in the distribution does never reach any of these algorithms since Normaliz recognizes parallelotopes $P$, computes the volume of a ``corner simplex'' and multiplies it by $n!$, $n=\dim  P$. This takes $\sim 0.01$ s, even for $n=28$. It would certainly be possible to go to $n=32$ with the Lawrence algorithm or descent with isomorphism types.

In \cite{Descent} the reader can find performance data for the descent algorithm applied to \verb|cu-20| and \verb|cu-24|. In addition, more general parallelotopes of the same dimensions are computed there. The computation  times show that the arithmetic is secondary and the times are dominated by the combinatorial complexity.

RAM usage for \verb|cu-20| is $\sim 1.1$ GB for both algorithms, and for \verb|cu-24| we need $23$ GB and $2$ GB. \verb|cu-28| takes $428$ GB and $127$ GB. That signed decomposition is so much better, is due to the fact that it does not store the extreme rays.

\subsubsection{Polytopes from social choice}
Computational data for polytopes from social choice are contained in \cite{BIS2} for the primal algorithm and symmetrization and in  \cite{Descent} for descent in the  face lattice. Voting schemes with $5$ candidates are essentially inaccessible to them. The Lawrence algorithm has now reached them, and \cite{5cand} contains data for them. Tables \ref{tab_combi5} \ref{tab_5cand} are imported from there. The names of the polytopes are explained in \cite{5cand}.

For two polytopes we have included the number of extreme rays to show the order of magnitude. For the Lawrence algorithm they are not needed explicitly, and in particular they need not be stored.

\begin{table}[hbt]
\begin{tabular}{rrrrrr}
\midrule[1.2pt]
\strut                            & dim $C$ & \#ext &$\#$ supp & $\#$ tri dual &  $\#$ hollow   \\
\midrule[1.2pt]
\strut \ttt{strictBorda 4cand}    &   24    &   &       33       &            100,738 &         324,862 \\
\hline
\strut \ttt{CondEffAppr 4cand}     &   74    &    &     80       &              1,620,052 &         30,564,920 \\
\hline
\strut \ttt{Condorcet}            &  120    &  290,064   &     124       &            137,105 &       6,572,904 \\
\hline
\strut \ttt{PlurVsRunoff}         &  120    &  80,912,472   &    125       &          4,912,369 &      93,749,784 \\
\hline
\strut \ttt{CWand2nd}             &  120    &      &   126       &         15,529,730 &     608,572,514 \\
\hline
\strut \ttt{CondEffPlurRunoff}    &  120    &      &   127       &        246,310,369 &   5,456,573,880 \\
\hline
\strut \ttt{CondEffPlur}          &  120    &      &   128       &      2,388,564,481 & 39,390,184,920 \\
\midrule[1.2pt]
\end{tabular}
\vspace*{1ex}
\caption{Combinatorial data} \label{tab_combi5}
\end{table}

The stages (1)--(3) of all computations could be done by $64$ bit arithmetic, and this holds even for the volume computations of the first and third polytope. The volume computations of the two largest had to be done with fixed precision.

For the two largest examples it was necessary to use distributed computation on a high performance cluster (indicated by HPC). For this reason we have split the computation times. For \verb|CondEffPlur| on the HPC of the University of Osnabrück the time was $< 9$ h. We refer the reader to \cite{5cand} for more information. 
\begin{table}[hbt]
\begin{tabular}{rrrrr}
\midrule[1.2pt]
\strut                 & \multicolumn{1}{c}{RAM} & \multicolumn{3}{c}{time }   \\
\cline{3-5}
\strut                 &  \multicolumn{1}{c}{in GB}&  \multicolumn{1}{c}{stages (1) -- (3)} & \multicolumn{1}{c}{stage (4)} &  \multicolumn{1}{c}{total} \\
\midrule[1.2pt]
\strut \ttt{strictBorda 4cand}     &   0.35 &     1.278 s  &     0.464 s &    1.742 s \\
\hline
\strut \ttt{CondEffAppr 4cand}      &   7.4 &     97.8 s  &     14:31 m &    16:09 m \\
\hline
\strut \ttt{Condorcet}             &   1.67 &    18.0 s  &    52.493 s &     1:10 m \\
\hline
\strut \ttt{PlurVsRunoff}          &  26.2 &     12:40 m  &   1:29:21 h &  1:42:01 s \\
\hline
\strut \ttt{CWand2nd}              &  56.4 &     49:55 m  &  10:21:36 h & 11:11:31 h \\
\hline
\strut \ttt{CondEffPlurRunoff}     & 113 &  13:30:22 h  &   HPC          &  --- \\
\hline
\strut \ttt{CondEffPlur}           & 646    & 125:27:20 h   &   HPC         &  --- \\
\midrule[1.2pt]
\end{tabular}
\vspace*{1ex}
\caption{Memory usage and times for parallelized volume computations} \label{tab_5cand}
\end{table}

\end{document}